\documentclass{amsart}
\usepackage{verbatim, hyperref}
\usepackage{rotating} 
\usepackage{amssymb}

\newtheorem{theorem}{Theorem}[section]

\newtheorem{lemma}[theorem]{Lemma}

\theoremstyle{plain}
\numberwithin{equation}{theorem}

\theoremstyle{remark}

\newcommand{\F}{{\mathbb F}}
\newcommand{\Fp}{{\mathbb F}_p}
\newcommand{\Q}{{\mathbb Q}}

\newcommand{\cY}{{\mathcal Y}}

\newcommand{\fo}{\mathfrak o}

\newcommand{\fp}{\mathfrak p}

\newcommand{\Kbar}{\overline K}

\DeclareMathOperator{\bN}{\mathbb{N}}

\newcommand{\A}{{\mathbb A}}

\newcommand{\bP}{{\mathbb P}}

\newcommand{\bG}{{\mathbb G}}
\newcommand{\Zp}{{\mathbb{Z}_p}}

\newcommand{\bC}{{\mathbb C}}

\newcommand{\bQ}{{\mathbb Q}}

\newcommand{\lra}{\longrightarrow}

\newcommand{\cF}{\mathcal{F}}

\newcommand{\cU}{\mathcal{U}}

\newcommand{\cH}{\mathcal{H}}

\newcommand{\Cp}{\bC_p}
\newcommand{\PCp}{\bP^1(\bC_p)}

\newcommand{\Dbar}{\overline{D}}

\begin{document}

\title{Bounding periods of subvarieties of $(\bP^1)^n$}


\author{J.~P.~Bell}
\address{
Jason Bell\\
Department of Pure Mathematics\\
University of Waterloo\\
Waterloo, ON N2L 3G1\\
CANADA 
}
\email{jpbell@uwaterloo.ca}

\author{D.~Ghioca}
\address{
Dragos Ghioca\\
Department of Mathematics\\
University of British Columbia\\
Vancouver, BC V6T 1Z2\\
Canada
}
\email{dghioca@math.ubc.ca}

\author{T.~J.~Tucker}
\address{
Thomas Tucker\\
Department of Mathematics\\
University of Rochester\\
Rochester, NY 14627\\
USA
}
\email{ttucker@math.rochester.edu}

\begin{abstract}
Using methods of $p$-adic analysis, along with the powerful result of Medvedev-Scanlon \cite{MS-Annals} for the classification of periodic subvarieties of $(\bP^1)^n$,  we bound the length of the orbit of a periodic subvariety $Y\subset (\bP^1)^n$ under the action of a dominant  endomorphism.
\end{abstract}

\maketitle

\section{Introduction}

In \cite{MorSil-1}, Morton and Silverman conjecture that there is a
constant $C(N,d,D)$ such that for any morphism $f: \bP^N \lra \bP^N$
of degree $d$ defined over a number field $K$ with $[K:\bQ] \leq D$,
the number of preperiodic points of $f$ over $K$ is less than or equal
to $C(N,d,D)$.  This conjecture remains very much open, but in the
case where $f$ has good reduction at a prime $\fp$, a great deal has
been proved about bounds depending on $\fp$, $N$, $d$, $D$ (see
\cite{mike-thesis, Pezda, Hutz-0}).  

In this paper, we study the more general problem of bounding length of the orbit of periodic 
subvarieties of any dimension; we prove the following results.  

\begin{theorem}
\label{main result smooth point}
Let $n\in\mathbb{N}$, let $K$ be a finite extension of $\Q_p$, let
$\kappa$ be its residue field, and let $e$ be the ramification index
of $K/\Q_p$.  Let $f_1,\dots, f_n\in K(x)$ be rational functions of
good reduction modulo $p$ and let $\Phi: (\bP^1)^n \lra (\bP^1)^n$ be
the endomorphism given by
\begin{equation}
\label{form of dominant endomorphism smooth point}
(x_1,\dots, x_n)\mapsto (f_1(x_1),\dots, f_n(x_n)).
\end{equation} 
Let $Y$ be a subvariety of $(\bP^1)^n$, and assume there is a
nonsingular point $(\alpha_1,\dots, \alpha_n)\in Y(K)$. Then there
exists a constant $C$ depending on $n$, $p$, $e$, $[\kappa:\Fp]$ and
the degree of the rational functions $f_i$ (but independent of $Y$)
such that if $Y$ is periodic under the action of $\Phi$, then the
length of its orbit is bounded by $C$.
\end{theorem} 

We prove Theorem~\ref{main result smooth point} as a consequence of the next result.

\begin{theorem}
\label{main result}
Let $K$ be a finite extension of $\Q_p$,  let $\kappa$ be its residue field and let $e$ be the ramification index of $K/\Q_p$. 
Let $f_1,\dots, f_n\in K(x)$ be rational functions of good reduction modulo $p$  and let $\Phi: (\bP^1)^n \lra (\bP^1)^n$ be the endomorphism given by 
\begin{equation}
\label{form of dominant endomorphism}
(x_1,\dots, x_n)\mapsto (f_1(x_1),\dots, f_n(x_n)).
\end{equation} 
Let $Y$ be a subvariety of $(\bP^1)^n$, and assume there is a point $(\alpha_1,\dots, \alpha_n)\in Y(K)$ such that for each $i=1,\dots, n$, we have that  $\alpha_i$ is a nonpreperiodic point for $f_i$. Then there exists a constant $C$ depending on $n$, $p$, $e$, $[\kappa:\Fp]$, and the degree of the rational functions $f_i$  (but independent of $Y$) such that if $Y$ is periodic under the action of $\Phi$, then the length of its
orbit is bounded
by $C$. 
\end{theorem}

We note that any dominant (regular) endomorphism $\Phi$ of $(\bP^1)^n$ is split, i.e., it is of the form 
$$(x_1,x_2,\dots, x_n)\mapsto \left(f_1(x_{\sigma(1)}),f_2(x_{\sigma(2)}),\dots, f_n(x_{\sigma(n)})\right),$$
where the $f_i$'s are rational functions, while $\sigma$ is a permutation of $\{1,\dots, n\}$. So, at the expense of replacing $\Phi$ by an iterate $\Phi^k$ (for some positive integer $k$ dividing $n!$), each dominant endomorphism of $(\bP^1)^n$ is of the form \eqref{form of dominant endomorphism}. 

In \cite{BGT-IMRN}, a similar result was obtained bounding the length of the orbit of any periodic subvariety (with a distinguished smooth $p$-adic point of good reduction) under the action of an \'etale endomorphism of an arbitrary variety (not only $(\bP^1)^n$). However, the presence of ramification as it is in our case adds new complications. Also, we note that in \cite{GN}, using results regarding the decomposition of polynomials, it was shown that the length of the orbit of a periodic subvariety of $\A^n$ under the coordinatewise action of $n$ one-variable polynomials is uniformly bounded by a constant depending only on $n$ and on the degrees of the polynomials. It is important to point out that the method from \cite{GN} cannot be extended to treat the case of rational functions; essentially, for rational functions, one lacks all the sharper results regarding their decomposition, as established in \cite{MS-Annals} for the case of polynomials.  

In our Theorem~\ref{main result}, the constant $C$ bounding the length of the orbit of a periodic subvariety of $(\bP^1)^n$ depends on the degrees of the rational functions $f_i$ and also on the good reduction data associated to our problem, which refers to the knowledge about a single point on $Y$. So, from this point of view, our result is less restrictive than the one obtained in  \cite{Hutz-1}, where a similar bound as the one from our Theorem~\ref{main result} is established under the additional assumption that also the periodic subvariety $Y$ (along with its iterates under $\Phi$) have good reduction modulo $p$, i.e., reducing modulo $p$ yields a  subvariety $\bar{Y}$, which is an irreducible subvariety defined over $\overline{\Fp}$ of same degree as $Y$ (and a similar property holds for each iterate of $Y$ under $\Phi$). Even though the result from \cite{Hutz-1} is valid for more general dynamical systems (i.e., endomorphisms of $\bP^n$ as opposed to endomorphisms of $(\bP^1)^n$), the additional hypothesis regarding the good behaviour of the periodic subvariety modulo $p$ makes it more restrictive than our Theorem~\ref{main result}.

Theorem~\ref{main result} is proved using a $p$-adic
analytic parametrization of forward orbits under the action of a rational function (such as in \cite{gap-compo}), combined with the description of periodic subvarieties of $(\bP^1)^n$ (obtained by Medvedev-Scanlon \cite{MS-Annals}).  

We start by describing in Section~\ref{section setup} the various useful results regarding the $p$-adic parametrization of orbits of points under the action of a rational function. We continue by proving in Section~\ref{section curves} the special case of Theorem~\ref{main result} when the periodic subvariety is a curve contained in $(\bP^1)^2$ (see Theorem~\ref{main result curve}). We conclude our paper by reducing the general Theorem~\ref{main result} to the special Theorem~\ref{main result curve}, using the description of periodic subvarieties of $(\bP^1)^n$ provided by Medvedev-Scanlon in \cite{MS-Annals}; then we derive the conclusion of Theorem~\ref{main result smooth point} as a consequence of Theorem~\ref{main result}.


\section{Useful results}
\label{section setup}

In this Section, we follow the setup and the results from \cite[Section~2]{gap-compo}; see also \cite[Chapter~6]{book}.  
For a prime $p$, we let $\Zp$  denote the ring of
$p$-adic integers, $\Q_p$ will denote the field of $p$-adic
rationals, and $\Cp$ will denote the completion of an algebraic
closure of $\Q_p$ with respect to the $p$-adic absolute value $|\cdot |_p$. 
Given a point $y\in\Cp$ and a real number $r>0$, write
$$D(y,r) = \{x\in\Cp : |x-y|_p < r\}, \quad
\Dbar(y,r) = \{x\in\Cp : |x-y|_p \leq r\}$$
for the open and closed disks, respectively, of radius $r$
about $y$ in $\Cp$. We write $[y]\subseteq \PCp$ for the residue class of a point $y\in\PCp$.
That is, $[y]=D(y,1)$ if $|y|\leq 1$,
or else $[y]=\PCp\setminus\Dbar(0,1)$ if $|y|>1$.

Let $K$ be a finite extension of $\Q_p$, let $\fo_K$ be its ring of $p$-adic integers, let $\pi$ be a uniformizer for $\fo_K$ and let $\kappa$ be its corresponding residue field.  Let $f\in K(x)$ be a rational function of degree $d\ge 1$. We say that $f$ has good reduction modulo $\pi$ (or equivalently modulo $p$) if writing the rational function $f$ as $$f([X_0:X_1])=[f_0(X_0,X_1):f_1(X_0,X_1)]$$
for some coprime homogeneous polynomials $f_0,f_1\in \fo_K[X_0,X_1]$
of degree $d$ for which at least one coefficient is a $\pi$-adic unit,
then reducing the coefficients of both $f_0$ and $f_1$ modulo $\pi$
(or equivalently, modulo $p$)  yields two coprime homogeneous
polynomials in $\kappa[X_0,X_1]$ of same degree $d$.  Note that our
choice of coordinates here induces a choice of model
$\bP_{\fo_K}^1$ and a morphism $F:  \bP_{\fo_K}^1 \lra
\bP_{\fo_K}^1$ of $\fo_K$-schemes such that $(F)_K = f$ when $f$
has good reduction at $\pi$.

The action of a $p$-adic power series
$f\in\fo_K[[z]]$ on $D(0,1)$ 
is either attracting
(i.e., $f$ contracts distances)
or quasiperiodic
(i.e., $f$ is distance-preserving),
depending on its linear coefficient.
Rivera-Letelier
gives a more precise description of this dichotomy
in \cite[Sections~3.1 and~3.2]{Riv1}. 
The following two Lemmas follow from \cite[Propositions~3.3~and~3.16]{Riv1} and they were stated in \cite[Lemmas~2.1~and~2.2]{gap-compo} in the case $K=\Q_p$; the general case follows by an identical argument (for more details regarding the applications to arithmetic dynamics of the $p$-adic parametrization of orbits, see \cite{book}).

\begin{lemma}
\label{lem:attr}
Let $K$ be a finite extension of $\Q_p$, let $\fo_K$ be its ring of $p$-adic integers, and let $\pi$ be a uniformizer for $\fo_K$. We let $f(z)=a_0+a_1 z + a_2 z^2 + \cdots \in\fo_K[[z]]$ be a
nonconstant power series with $|a_0|_p, |a_1|_p < 1$.
Then there is a point $y\in \pi\cdot \fo_K$ such that
$f(y)=y$, and  $\lim_{n\rightarrow\infty}f^n(z)=y$
for all $z\in D(0,1)$.
Write $\lambda=f'(y)$; then
$|\lambda|_p < 1$,
and:

\begin{enumerate}
\item (Attracting).
If $\lambda\neq 0$, then there is a radius $0<r<1$ and
a power series $u\in K[[z]]$ mapping
$\Dbar(0,r)$ bijectively onto $\Dbar(y,r)$
with $u(0)=y$, such that 
for all $z\in D(y,r)$ and $n\geq 0$,
$$f^n(z) = u(\lambda^n u^{-1}(z)).$$

\item (Superattracting).
If $\lambda=0$, then write $f$ as
$$f(z) = y + c_m (z-y)^m + c_{m+1} (z-y)^{m+1} + \cdots \in\fo_K[[z-y]]$$
with $m\geq 2$ and $c_m\neq 0$.
If $c_m$ has an $(m-1)$-st root in $K$ (or equivalently, in $\fo_K$), then 
there are radii 
$0<r,s < 1$ and a power series $u\in K[[z]]$ mapping
$\Dbar(0,s)$ bijectively onto $\Dbar(y,r)$
with $u(0)=y$, such that
for all $z\in D(y,r)$ and $n\geq 0$,
$$f^n(z) = u\Big( (u^{-1}(z))^{m^n}\Big).$$
\end{enumerate}
\end{lemma}

\begin{proof}
The proof is identical to the proof of \cite[Lemma~2.1]{gap-compo}; see also \cite{Silverman-new} for a similar result in the superattracting case.  
\end{proof}

The next result is a slight generalization of \cite[Lemma~2.2]{gap-compo}.

\begin{lemma}
\label{lem:siegel}
Let $K$, $\fo_K$ and $\pi$ be as in Lemma~\ref{lem:attr} and let $\kappa$ be the residue field of $K$. Let $f(z)=a_0+a_1 z + a_2 z^2 + \cdots \in\fo_K[[z]]$ be a
nonconstant power series with $|a_0|_p<1$ but $|a_1|_p = 1$.
Then for any nonperiodic $x\in \pi\cdot\fo_K$,
there are: an integer $k\geq 1$ depending only on $p$, $e$ and $[\kappa:\Fp]$,
radii $0<r< 1$ and $s\geq |k|_p$,
and a power series $u\in K[[z]]$ 
mapping $\Dbar(0,s)$ bijectively onto $\Dbar(x,r)$
with $u(0)=x$,
such that for all $z\in \Dbar(x,r)$ and $n\geq 0$,
$$f^{nk}(z) = u(nk+ u^{-1}(z)).$$
\end{lemma}

\begin{proof}
The conclusion of Lemma~\ref{lem:siegel} (including the fact that $k$ is bounded solely in terms of $p$, $e$ and $[\kappa:\Fp]$) follows by combining \cite[Proposition~2.1]{BGT-IMRN} with \cite{Poonen}.   
\end{proof}


\section{The case of curves}
\label{section curves}

We first prove Theorem~\ref{main result} when $Y$ is a curve in $\bP^1\times \bP^1$.
\begin{theorem}
\label{main result curve}
Let $K$ be a finite extension of $\Q_p$, let $\kappa$ be its residue field and let $e$ be the ramification index of $K/\Q_p$. 
Let $f_1,f_2\in K(x)$ be rational functions of good reduction modulo $p$ and let $\Phi: (\bP^1)^2 \lra (\bP^1)^2$ be the endomorphism given by $(x_1,x_2)\mapsto (f_1(x_1),f_2(x_2))$.  
Let $Y$ be an irreducible, periodic curve of $(\bP^1)^2$ defined over $K$, and assume there is a point $\alpha$ on $Y(K)$ which is  not preperiodic under the action of $\Phi$. Then there exists a constant $c$ depending on the degress of the rational functions $f_i$ and also depending on $p$, $e$ and $[\kappa:\Fp]$, but independent of $Y$, such that the length of the 
orbit of $Y$ under $\Phi$ is bounded
by $c$.
\end{theorem}

\begin{proof}
We let $\alpha:=(\alpha_1,\alpha_2)\in \bP^1_K\times\bP^1_K$. If for some $i=1,2$, we have that $\alpha_i$ is preperiodic for $f_i$, then since $\alpha$ is \emph{not} preperiodic under the action of $\Phi$ but $Y$ is periodic under the action of $\Phi$, we conclude that the curve $Y$ projects to $\alpha_i$ on the $i$-th coordinate axis. Therefore, the length of the orbit of $Y$ is the length of the orbit of $\alpha_i$ (under the action of $f_i$). Because $\alpha_i\in \bP_K^1$ is a periodic point for $f_i$ and moreover, $f_i$ has good reduction modulo $p$, then the length of the orbit of $\alpha_i$ is bounded only in terms of $p$, $e$, $[\kappa:\Fp]$, as proven in \cite{mike-thesis, Hutz-0}; hence the length of the orbit of $Y$ is bounded by a constant which is independent of $Y$.

So, from now on, we may assume that each $\alpha_i$ (for $i=1,2$) is not preperiodic. 

At the expense of replacing $K$ by a finite extension of degree bounded solely in terms of the degrees of the rational functions $f_1$ and $f_2$ (and therefore replacing $e$ and $[\kappa:\Fp]$ by multiples of them bounded by a quantity depending again only in terms of the degrees of the $f_i$'s), we may assume the following. For some $i=1,2$ if $\alpha_i$ modulo $\pi$ is in the  cycle of a periodic critical point $Q$ of $f_i$ of period $\ell$, given a local coordinate $x_{i,Q}$ at $Q$, we write
\begin{equation}
\label{superattracting 1}
f_i^\ell(x_{i,Q})= c_{i,Q}x_{i,Q}^{m_{i,Q}} + O\left(x_{i,Q}^{m_{i,Q}+1}\right),
\end{equation}
where $m_{i,Q}\ge 2$ and $c_{i,Q}\ne 0$; then $K$ contains all the $(m_{i,Q}-1)$-st roots of $c_{i,Q}$. Since there are at most finitely many superattracting periodic points, then we are guaranteed that a finite extension of $K$ suffices. If there is no superattracting periodic point modulo $p$ in the orbit of $\alpha_i$ under the action of $f_i$ modulo $\pi$, then we do not need to extend $K$.

We let $\fo_K$ be the subring of $p$-adic integers in $K$, and we let $\pi$ be a uniformizer for $\fo_K$.

We proceed as in the proof of \cite[Theorem~1.4]{gap-compo} (see also \cite[Chapter~11]{book}); our goal is to show that there exists a positive integer $k$ depending only on data coming from $K$ (i.e, its ramification index and also the size of its residue field) and also there exists a positive integer $\ell$ (which may depend on the $\alpha_i$'s) such that for each $i=1,2$ and for each $0\le j\le k-1$ there exist positive integers $m_i$ and  $p$-adic analytic power series $F_{i,j}\in K[[z_0,z_1,z_{m_i}]]$ such that (modulo linear transformations) we have $p$-adic analytic parametrizations: 
\begin{equation}
\label{p-adic parametrization 1}
f_i^{\ell+j+nk}(\alpha_i)=F_{i,j}\left(n,\pi^n, \pi^{m_i^n}\right)\text{ for each }n\ge 0.
\end{equation}
The integers $m_i$ from \eqref{p-adic parametrization 1} are the same as the integers $m_{i,Q}$ from equation \eqref{superattracting 1}; if $\alpha_i$ is not superattracting, then actually $F_{i,j}$ is a $p$-adic analytic power series only in the variables $z_0$ and $z_1$, or alternatively, $m_i=1$ in that case. 

In order to derive the $p$-adic parametrization \eqref{p-adic parametrization 1}, we proceed as in the proof of \cite[Theorem~1.4]{gap-compo}; in particular, we split our analysis into various steps.

\textbf{Step (i).} 
First we note that there exist positive integers $\ell_0$ and $k_0$ such that the residue class of $f_i^{\ell_0}(\alpha_i)$ (for $i=1,2$) is fixed under the induced action on $\bP_{\kappa}^1$ of the reduction of $f_i^{k_0}$ modulo $\pi$; here we use the fact that each $f_i$ has good reduction modulo $\pi$. Furthermore, note that both integers $k_0$ and $\ell_0$ depend only on $p$ and on $[\kappa:\Fp]$; for more details,  see \cite[Step~(ii),~p.~1063]{gap-compo}.   By a ${\rm PGL}(2,\fo_K)$-change of coordinates for each $i=1,2$, we may assume that $f_i^{\ell_0}(\alpha_i)\in \pi\fo_K$, and therefore $f_i^{k_0}$ may be written as a nonconstant power series in $\fo_K[[z]]$ mapping $D(0,1)$ into itself. 

\textbf{Step (ii).} 
Let $i\in\{1,2\}$. If $\left|f_i^{k_0}\left(f_i^{\ell_0}(\alpha_i)\right)\right|_p<1$, then Lemma~\ref{lem:attr} yields that there is a
point $y_i\in D(0,1)$ fixed by $f_i^{k_{0}}$  
along with radii $r_i$ and $s_i$ (where $s_i := r_i$
in the non-superattracting case), and an associated
power series $u_i\in K[[z]]$.
Set $k_{1,i}=k_{0}$ and $\ell_{1,i}=\ell_{0}+ n_i k_{1,i}$
for a suitable integer $n_i\geq 0$ so that
\begin{equation}
\label{in the disk}
f_i^{\ell_{1,i}}(\alpha_i)\in\Dbar(y_i,r_i). 
\end{equation} 
Define $\lambda_{i}:=(f_i^{k_{0}})'(y_i)$ 
to be the multiplier of the point $y_i$; then in this case, 
$|\lambda_{i}|_p<1$. 
Define $\mu_i:=u_i^{-1}(f_i^{\ell_{1,i}}(\alpha_i))$; note that
$\mu_i\in \pi\cdot \fo_K$, because $s_i<1$.  In addition,
$\mu_i\neq 0$, because $u_i$ is bijective
and $u_i(0)=y_i$ is fixed by $f_i^{k_0}$,
while $u_i(\mu_i)=f_i^{\ell_{1,i}}(\alpha_i)$ is not fixed by $f_i^{k_0}$. 

If
$|(f_i^{k_{0}})'(f_i^{\ell_{0}}(\alpha_i))|_p=1$, we  
apply Lemma~\ref{lem:siegel} to $f_i^{k_{0}}$ and the point
$f_i^{\ell_{0}}(\alpha_i)$ to
obtain radii $r_i$ and $s_i$ and a power series
$u_i$.  Define $\mu_i:=u_i^{-1}(f_i^{\ell_{0}}(\alpha_i))$,
and set $\ell_{1}=\ell_{0}$
and $k_{3,i}=n_i k_{0}$, for a suitable integer $n_i\geq 1$ depending only on $p$, $e$ and $[\kappa:\Fp]$ (see Lemma~\ref{lem:siegel} in which the integer $k$ from its conclusion depends only on $p$, $e$ and $[\kappa:\Fp]$) 
so that $f_i^{k_{3,i} + \ell_{1}}(\alpha_i)\in\Dbar(f_i^{\ell_{1}}(\alpha_i),r_i)$.
(Following the notation from \cite{gap-compo}, the jump from
a subscript of $0$ to $3$ is because certain complications,
to be addressed in Steps~(iii) and~(iv),
do not arise in the quasiperiodic case.)
Note that
$f_i^{\ell_{1,i} + n k_{3,i}}(\alpha_i)$ may be expressed as a
power series in the integer $n\geq 0$; specifically,
$f_i^{\ell_{1,i} + n k_{3,i}}(\alpha_i)=u_i(n k_{3,i}+\mu_i)$.

{\bf Step (iii).}
In this step, we consider only the case that
$0<|\lambda_{i}|_p<1$
(i.e., attracting but not superattracting).
We will express certain functions of $n$
as power series in $n$ and $\pi^n$ and thus obtain a $p$-adic parametrization of the orbit of $\alpha_i$ under $f_i$ as in  \eqref{p-adic parametrization 1}, where the corresponding $p$-adic analytic function $F_{i,j}$ is actually a function of only $2$ variables (i.e., $m_i=1$ with the notation as in \eqref{p-adic parametrization 1}).

Write $\lambda_{i} = \alpha_i \pi^{e_{i}}$, where $e_{i}\geq 1$
and $\alpha_i\in\fo_K^*$.
If $\alpha_i$ is a root of unity, we can
choose an integer $M_{1,i}\geq 1$ such that $\alpha_i^{M_{1,i}}=1$.
If $\alpha_i$ is not a root of unity,
it is well known that
there is an integer $M_{1,i}\geq 1$ such
that $\alpha_i^{n M_{1,i}}$
can be written as a power series in $n$
with coefficients in $\fo_K$. Furthermore, $M_{1,i}$ depends only on $p$, $e$ and $[\kappa:\Fp]$, which can be seen also applying Lemma~\ref{lem:siegel} to the function $z\mapsto \alpha_iz$; for more details, see \cite[Step~(iii)]{gap-compo}. We set
\begin{equation}
\label{k 3 1}
k_{3,i}:= M_{1,i} k_{1,i};
\end{equation}
also, we replace $\lambda_i$ by $\lambda_{i}^{M_{1,i}}$ and replace $e_i$ by  $M_{1,i} \cdot e_{i}$. 
(The subscript again jumps to $3$ in \eqref{k 3 1} because of
the complications of Step~(iv).)
Thus, we can write
\begin{equation}
\label{eq:gdef1}
\lambda_{i}^n = (\pi^n)^{e_{i}} g_{1,i}(n)
\qquad \text{for all integers } n\geq 0,
\end{equation}
for some power series $g_{1,i}(z)\in\fo_K[[z]]$.

{\bf Step (iv).}
In this step, we consider only the
superattracting case, that $\lambda_{i}=0$,
and
we will express certain functions of $n$
as power series in $n$, $\pi^n$, and $\pi^{m_{i}^n}$,
where $m_{i}\geq 2$ is a certain integer (see \eqref{superattracting 1}).

We let first the integer $m_{i,0}$ be the order of the unique
superattracting point of $f_i$ in $D(0,1)$, as in
Lemma~\ref{lem:attr}(ii). Then we write $m_{i,0}:=a_i \pi^{b_i}$,
for some $a_i\in\fo_K^*$ and some integer $b_i\geq 0$.
Then as in Step~(iii),
we can find a positive integer $M_{1,i}$ depending only on $p$, $e$ and $[\kappa:\Fp]$ such that
$a_i^{nM_{1,i}}$ 
can be written as a power series in $n$ with coefficients in $\fo_K$.
Set
$$k_{2,i}:=M_{1,i} \cdot k_{1,i}
\quad \text{and} \quad
m_{i,1}:=m_{i,0}^{M_{1,i}}.$$
Then $m_{i,1}^n$
can be written as a power series in $n$ and $\pi^n$, with coefficients in $\fo_K$.

In addition, recall that $\mu_i=u_i^{-1}(f_i^{\ell_{i}}(\alpha_i))$
satisfies $0<|\mu_i|_p<1$; thus,
we can write $\mu_i=\beta_i \pi^{e_i}$,
where $e_i\geq 1$ and $\beta_i\in\fo_K^*$.
If $\beta_i$ is a root of unity with, say, $\beta_i^{M_{2,i}}=1$
for some positive integer $M_{2,i}$ (depending only on $p$, $e$ and $[\kappa:\Fp]$, since the size of the group of roots of unity contained in $K$ depends only on $p$, $e$ and $[\kappa:\Fp]$), choose a positive integer $M_{3,i}$ (again depending only on $p$, $e$ and $[\kappa:\Fp]$)
so that $M_{2,i} | \big( m_{i,1}^{2M_{3,i}} - m_{i,1}^{M_{3,i}} \big)$.
Set
$$k_{3,i} := M_{3,i} k_{2,i}\text{ and }m_{i,2}:=m_{i,1}^{M_{3,i}}$$
and note that $\beta_i^{m_{i,2}^{n}}$ is constant in $n$.

On the other hand, if $\beta_i$ is not a root of
unity, then as in Step~(iii),
there is an integer $M_{2,i}'$ depending only on $p$, $e$ and $[\kappa:\Fp]$ such that $\beta_i^{n M_{2,i}'}$
can be written as a power series in $n$ with coefficients in $\fo_K$.
As above, choose a positive integer $M_{3,i}'$ (depending only on $p$, $e$ and $[\kappa:\Fp]$) such that
$M_{2,i}' | \big( m_{i,1}^{2M_{3,i}'} - m_{i,1}^{M_{3,i}'} \big)$,
and set
$$k_{3,i} := M_{3,i}' k_{2,i}\text{ and also, } m_{i,2}:=m_{i,1}^{M_{3,i}'}.$$  
Then $m_{i,2}^n\equiv m_{i,2}\pmod{M_{2,i}'}$ for all positive integers $n$, and therefore
$$\beta_i^{m_{i,2}^{n}}
=\beta_i^{m_{i,2}}\cdot \beta_i^{m_{i,2}^n- m_{i,2}}
=\beta_i^{m_{i,2}}\cdot
\Big(\beta_i^{M_{2,i}'}\Big)^{\left(m_{i,2}^n-m_{i,2}\right)/M_{2,i}'}$$ 
can be written as a power series
in $( m_{i,2}^n-m_{i,2} )/M_{2,i}'$ with coefficients in $\fo_K$.
Using the fact that $p\nmid M_{2,i}'$, and expressing
$m_{i,2}^n = (m_{i,1}^n)^{M_{3,i}'}$ as a power series in $n$ and $\pi^n$
with coefficients in $\fo_K$, we conclude that
$\beta_i^{m_{i,2}^{n}}$ can in fact
be written as a power series in $n$ and $\pi^n$,
with coefficients in $\fo_K$.

Thus, whether or not $\beta_j$ is a root of unity, we can write
\begin{equation}
\label{eq:gdef2}
\mu_i^{m_{i,2}^n} = \left( \pi^{m_{i,2}^n} \right)^{e_i} g_{1,i}(n,\pi^n)
\qquad \text{for all integers } n\geq 0,
\end{equation}
for some power series $g_{1,i}(z_0,z_1)\in\fo_K[[z_0,z_1]]$.

{\bf Step (v).}
Let $k$ be the least common multiple of $k_{3,1}$ and of $k_{3,2}$; then $k$ is bounded by a constant depending only on $p$, $e$ and $[\kappa:\Fp]$.

Note that replacing
each $k_{3,i}$ by $k$ does not change the conclusions of our previous 
Steps since the radii $r_i$, $s_i$, the power series $u_i$,
the integer $\ell_{1,i}$, and the point $\mu_i$ are unaffected when we replace $k$ by a multiple of it. Moreover, in the quasiperiodic case,
$f_i^{k + \ell_{1,i}}(\alpha_i)$ still lies in $\Dbar(f_i^{\ell_{1,i}}(\alpha_i),r_i)$.

In the attracting case, we replace 
$\lambda_{i}$ by $\lambda_{i}^{k/k_{3,i}}$, also replace 
$e_{i}$ by $\frac{k}{k_{3,i}} e_{i}$, and let  
$$g_{2,i}(z):= \big( g_{1,i}(z) \big)^{k/k_{3,i}}\in \fo_K[[z]];$$ 
and in the superattracting case, we replace 
$m_{i}$ by $m_{i}^{k/k_{3,i}}$ and let 
$$g_{2,i}(z_0,z_1):=
g_{1,i}\Big(\frac{k}{k_{3,i}} z_0 , z_1^{k/k_{3,i}}\Big)\in \fo_K[[z_0,z_1]].$$ 
With this new notation, it follows from Steps~(i)--(iv) that for
any integer $n\geq 0$,
\begin{enumerate}
\item[(1)]
$f_i^{\ell_{1,i}+nk}(\alpha_i)=u_i(nk+\mu_i)$,
if $f_i^{\ell_{1,i}}(\alpha_i)$ lies in a quasiperiodic residue class;
\item[(2)]
$f_i^{\ell_{1,i}+nk}(\alpha_i)=u_i\big(\lambda_{i}^n \mu_i\big)
= u_i\big( (\pi^n)^{e_{i}} g_{2,i}(n) \mu_i \big)$,
if $f_i^{\ell_{1,i}}(\alpha_i)$ lies in an attracting residue class; and
\item[(3)] 
$f_i^{\ell_{1,i}+nk}(\alpha_i)=u_i\big(\mu_i^{m_{i}^n}\big)
= u_i\Big( \big(\pi^{m_{i}^n}\big)^{e_i} g_{2,i}(n,\pi^n) \Big)
$, 
if $f_i^{\ell_{1,i}}(\alpha_i)$ lies in a superattracting residue class,
\end{enumerate}
where
$\mu_i = u_i^{-1}(f_i^{\ell_{1,i}}(\alpha_i))$ is as in Step~(ii).
In particular, in all three cases, we have
expressed $f_i^{\ell_{1,i}+nk}(\alpha_i)$ as a power series
in $n$, $\pi^n$, and, if needed, $\pi^{m_{i}^n}$.

Let $L=\max\{\ell_{1,1},\ell_{1,2}\}$. 
For each $\ell=L,\ldots,L+k-1$ and each $i=1,2$,
choose a linear fractional transformation
$\eta_{i,\ell}\in {\rm PGL}(2,\fo_K)$ so that
$\eta_{i,\ell}\circ f_i^{\ell}(\alpha_i)\in D(0,1)$.
Then $\eta_{i,\ell}\circ f_i^{\ell-\ell_{1,i}}(D(0,1))\subseteq D(0,1)$,
because $f_i$ has good reduction.
Finally, define $E_{i,\ell}=\eta_{i,\ell}\circ f_i^{\ell-\ell_{1,i}}\circ u_i$,
so that $E_{i,\ell}\in K[[z]]$ maps $\Dbar(0,s_i)$
into $D(0,1)$.

{\bf Step (vi).}
In this step, we will write down power series $F_{i,\ell}$
for $f_i^{\ell + nk}$
in terms of $n$, $\pi^n$, and $\pi^{m_i^n}$.  We will also produce bounds
$B_{i,\ell}$ to be used in applying \cite[Lemma~3.1]{gap-compo}.
For each $i=1,2$, we consider the three cases that
$f_i^{\ell}(\alpha_i)$ lies in a quasiperiodic,
attracting (but not superattracting),
or superattracting residue class for the function $f_i^k$.

In the quasiperiodic case, for each $\ell=L,\ldots,L+k-1$,
define the power series
$$F_{i,\ell}(z_0) = E_{i,\ell}(k z_0 + \mu_i )
\in K[[z_0]],$$
so that $F_{i,\ell}(n)=\eta_{i,\ell}\circ f_i^{\ell + nk}(\alpha_i)$
for all $n\geq 0$.
All coefficients of $F_{i,\ell}$ have absolute value at most $1$,
because $|k|_p,|\mu_i|_p\leq s_i$
and $E_{i,\ell}$ maps $\Dbar(0,s_i)$ into $D(0,1)$.
Hence, we set our bound $B_{i,\ell}$ to be
$B_{i,\ell}:=0$.

Second, in the attracting (but not superattracting) case,
for each $\ell=L,\ldots,L+k-1$,
define the power series
$$F_{i,\ell}(z_0,z_1) =
E_{i,\ell}\big(z_1^{e_{i}} g_{2,i}(z_0) \mu_i\big)
\in K[[z_0,z_1]],$$
where $E_{i,\ell}$ and $g_{2,i}$ are as in Step~(v),
so that
$F_{i,\ell}(n,\pi^n)=\eta_{i,\ell}\circ f_i^{\ell + nk}(\alpha_i)$
for all $n\geq 0$.

Still in the attracting (but not superattracting) case,
because $E_{i,\ell}$ maps $\Dbar(0,s_i)$ into $D(0,1)$,
there is some $B_{i,\ell}> 0$ such that for every $j\geq 0$,
the coefficient of $z^j$
in $E_{i,\ell}(z)$ has absolute value at most $p^{j B_{i,\ell} }$.
Recalling also that $g_{2,i}\in\fo_K[[z]]$ and $|\mu_i|_p<1$,
it follows that if we write
$F_{i,\ell}(z_0,z_1)=\sum_{j=0}^{\infty} h_j(z_0)z_1^j$
(where each $h_j\in K[[z]]$),
then for each $j\geq 0$, all coefficients of $h_j$ have absolute value
at most $p^{j B_{i,\ell} }$.

Third, in the superattracting case,
for each $\ell=L,\ldots,L+k-1$, define the power series
$$F_{i,\ell}(z_0,z_1,z_{m_i}) =
E_{i,\ell}\big( g_{2,i}(z_0,z_1) z_{m_i}^{e_i} \big)
\in K[[z_0,z_1,z_{m_i}]],$$
where $E_{i,\ell}$ and $g_{2,i}$ are as in Step~(v),
so that
$F_{i,\ell}(n,\pi^n,\pi^{m_{i}^{n}})=\eta_{i,\ell}\circ f_i^{\ell + nk}(\alpha_i)$
for all $n\geq 0$.

Still in the superattracting case,
because $E_{i,\ell}$ maps $\Dbar(0,s_i)$ into $D(0,1)$,
there is some $B_{i,\ell}> 0$ such that for every $j\geq 0$,
the coefficient of $z^j$
in $E_{i,\ell}(z)$ has absolute value at most $p^{j B_{i,\ell}}$.
Hence, if we write
$F_{i,\ell}(z_0,z_1,z_{m_{i}})
=\sum_{j_1,j_2\geq 0} h_{j_1,j_2}(z_0)z_1^{j_1}z_{m_{i}}^{j_2}$
(where each $h_{j_1,j_2}\in K[[z]]$),
then as before, since $g_{2,i}\in\fo_K[[z_0,z_1]]$,
all coefficients of $h_{j_1,j_2}$ have absolute value
at most $p^{j_2 B_{i,\ell} }\leq p^{B_{i,\ell}(j_1 + j_2)}$. Finally, set
$$B:=\max_{L\le \ell\le L+k-1}\left(B_{1,\ell}+B_{2,\ell}\right).$$

{\bf Step (vii).} 
For each $\ell=L,\ldots,L+k-1$,
let $\cH_{\ell}\in \fo_K[t_1,t_2]$ be
the polynomial 
generating the vanishing ideal of the curve $Y\subset (\bP^1)^2$ dehomogenized with respect to the coordinates
determined by
$(\eta_{1,\ell},\eta_{2,\ell})$.
Then we define
\begin{equation}
\label{G l 0 m 1 2}
G_{\ell}(z_0,z_1,z_{m_1},z_{m_2}) = \cH_\ell(F_{1,\ell},F_{2,\ell})
\in K[[z_0,z_1,z_{m_1},z_{m_2}]].
\end{equation}
If there is no superattracting point for $f_i$ lying in the orbit of $\alpha_i$ modulo $\pi$, then $m_i=1$, i.e., the above power series $G_{\ell}$ has one less variable. Also, if it happens that $m_1=m_2$, then again we need one variable less in \eqref{G l 0 m 1 2}. From now on, we assume both $\alpha_1$ and $\alpha_2$ land in superattracting cycles modulo $\pi$ (i.e., $m_1,m_2>1$) and also, assume $m_1\ne m_2$; this would be the most general case in \eqref{G l 0 m 1 2} since we would deal with a $p$-adic analytic function of $4$ variables (the other cases are similar, only simpler). So, without loss of generality, we may assume $m_2>m_1>1$.

In all cases, by construction,
$G_{\ell}(n,\pi^n,\pi^{m_1^n}, \pi^{m_2^n})$
is defined for all integers $n\geq 0$, and moreover, 
\begin{equation}
\label{equivalent condition 1}
G_{\ell}(n,\pi^n,\pi^{m_1^n}, \pi^{m_2^n})=0\text{ if and only if }\Phi^{\ell+nk}(\alpha)\in Y.
\end{equation}

Now, if $G_{\ell}=0$ for some $\ell=L,\cdots, L+k-1$, then we get that $\Phi^{\ell+nk}(\alpha)\in Y$ for all $n\ge 0$ and since $\alpha$ is not preperiodic, while $Y$ is a curve, we conclude that $Y$ must be fixed under the action of $\Phi^k$, as desired (note that $k$ is bounded solely in terms of the degrees of the $f_i$'s and also on $p$, $e$ and $[\kappa:\Fp]$). 

So, from now on, we assume that $G_\ell\ne 0$ for each $\ell=L,\cdots, L+k-1$. Then for each $G_{\ell}$, we write
$$G_{\ell}(z_0,\pi^n,\pi^{m_1^n}, \pi^{m_2^n})
=\sum_{w\in\bN^3} g_w(z_0)\pi^{f_w(n)},$$
where for each $w:=(w_1,w_2,w_3)\in \bN^3$, we have that $f_w(n):=w_1n+w_2m_1^n+w_3m_2^n$. Then we let $v$ be the smallest element of $\bN^3$ with respect to the usual lexicographic order on $\bN^3$ such that $g_v\ne 0$. Also, for each element $w:=(w_1,w_2,w_3)\in \bN^3$, we define $|w|:=w_1+w_2+w_3$; for more details, see \cite[Section~3]{gap-compo}.

By our choice of the bound $B$ in Step~(vi),
and because all coefficients of $\cH_\ell$ lie
in $\fo_K$, all coefficients of $g_w$ have absolute value at most $p^{B|w|}$,
for every $w\in\bN^3$.
Since $G_{\ell}\left(n,\pi^n,\pi^{m_1^n}, \pi^{m_2^n}\right)$
is defined at every $n\geq 0$, then $g_v$ must converge
on $\Dbar(0,1)$; therefore, we may
choose a radius $0<s_{\ell}\leq 1$ for $g_v$ satisfying the hypotheses of 
\cite[Lemma~3.1]{gap-compo}. More precisely, $g_v$ has only finitely many zeros in $\Dbar(0,1)$ and so, we let $0<s_{\ell}\leq 1$ be the minimum distance
between any two distinct such zeros.

Let $s$ be the minimum of all the
$s_{\ell}$ as we vary $\ell\in \{L,\cdots, L+k-1\}$.
The set $\fo_K$ may be covered by finitely many disks
$D(\gamma,s)$ (for some points $\gamma\in\fo_K$). We let $N:=k\cdot p^M$ for some sufficiently large integer $M$ such that $|p^M|_p<s$.

Since each $G_\ell$ is nontrivial, we may apply \cite[Lemma~3.1]{gap-compo} 
(with the bound $B$ from Step~(vi) and radius $s$ from the previous
paragraph) and conclude that for each $j\in\{1,\dots, N\}$, the set 
\begin{equation}
\label{sparse set 1}
\{n\in\bN\colon \Phi^{j+nN}(\alpha)\in Y\}\text{ has natural density $0$.}
\end{equation}
In order to derive \eqref{sparse set 1}, we use \eqref{equivalent condition 1} and \cite[Lemma~3.1~and~Corollary~1.5]{gap-compo}. However, since $Y$ is periodic, the set
$$\{n\in \bN\colon \Phi^n(\alpha)\in Y\}\text{ has positive natural density,}$$
which contradicts \eqref{sparse set 1}. In conclusion, it must be that some $G_\ell$ is trivial and therefore $Y$ must be fixed by $\Phi^k$.

This concludes the proof of Theorem~\ref{main result curve}.
\end{proof}


\section{Proof of our main results}
\label{section proof}

We first prove Theorem~\ref{main result}. Our strategy is to prove separately its conclusion assuming that either
\begin{enumerate}
\item[(1)] all rational functions $f_i$ are \emph{disintegrated}, using the terminology from \cite{MS-Annals} (or \emph{non-special}, using the terminology from \cite{GN-Advances, GN}), i.e., for each $i=1,\dots, n$, we have that $\deg(f_i)\ge 2$  and also, $f_i(x)$ is not conjugated (through a linear transformation) to $x^{\pm \deg(f_i)}$, or to $\pm C_{\deg(f_i)}$ (where, for each $m\ge 2$, $\pm C_m(x)$ is the $m$-th Chebyshev polynomial, which satisfies $C_m\left(x+\frac{1}{x}\right)=x^m+\frac{1}{x^m}$), or to a Latt\'es map (i.e., the quotient of an endomorphism of an elliptic curve); or
\item[(2)] all rational functions $f_i$ are special, i.e., for each $i=1,\dots, n$, we have that $\deg(f_i)\ge 2$ and also, $f_i(x)$ is conjugated either to $x^{\pm \deg(f_i)}$, or to $\pm C_{\deg(f_i)}$, or to a Latt\'es map; or 
\item[(3)] each rational function $f_i$ has degree equal to $1$.
\end{enumerate}

In each of the above cases (1)-(3), we show that the length of the period of a periodic subvariety (satisfying the hypotheses of Theorem~\ref{main result}) is uniformly bounded. Case~(1) follows from combining our Theorem~\ref{main result curve} along with the description from \cite{MS-Annals} of the periodic subvarieties of $(\bP^1)^n$ under the coordinatewise action of $n$ disintegrated rational functions (for more details, see Theorem~\ref{main result disintegrated}). Case~(2) above follows using classical results regarding algebraic subgroups of product of tori and of elliptic curves (for more details, see Theorem~\ref{main result groups}). Finally, case~(3) follows using a direct analysis since for linear maps, one can find explicit formulas for the points in any of their orbits (for more details, see Theorem~\ref{main result trivial}). Then we finish the proof of Theorem~\ref{main result} by employing one more time the results of \cite{MS-Annals} which allows us to prove the conclusion of Theorem~\ref{main result} after splitting the action of the $f_i$'s in the above $3$ cases and using Theorems~\ref{main result disintegrated},~\ref{main result trivial}~and~\ref{main result groups}.

\begin{theorem}
\label{main result disintegrated}
Theorem~\ref{main result} holds under the additional assumption that each rational function $f_i$ has degree greater than $1$, and moreover, no $f_i(x)$ is conjugated to either $x^{\pm \deg(f_i)}$, or to $\pm C_{\deg(f_i)}$, or to a Latt\'es map.
\end{theorem}

\begin{proof}
Since each $f_i$ is disintegrated then, according to \cite[Proposition~2.21]{MS-Annals}, we know that each (irreducible) periodic subvariety $Y\subset (\bP^1)^n$ is an irreducible component of the intersection of finitely many hypersurfaces $Y_{i,j}$ of $(\bP^1)^n$, which are of the form $Y_{i,j}:=\pi_{i,j}^{-1}(C_{i,j})$, where $C_{i,j}$ is an irreducible periodic curve defined over $K$  under the action of $(f_i,f_j)$ on $\bP^1\times\bP^1$ and $\pi_{i,j}:(\bP^1)^n\lra \bP^1\times\bP^1$ is the projection on the $(i,j)$-th coordinate axes; note that $C_{i,j}$ is also defined over $K$ since the projection of the irreducible $K$-subscheme $Y$ of $(\bP^1)^n$ on any two coordinates would still be an irreducible $K$-subscheme (of $(\bP^1)^2$).  By Theorem~\ref{main result curve}, we know that the length of the orbit of $C_{i,j}$ under the induced action of $(f_i,f_j)$ on $\bP^1\times\bP^1$ is bounded independently of $C_{i,j}$ (and only depending on the degrees of $f_i$ and $f_j$ and also on $p$, $e$ and $[\kappa:\Fp]$); note that the point $(\alpha_i,\alpha_j)\in C_{i,j}(K)$ satisfies the hypotheses from Theorem~\ref{main result curve}. Therefore, each hypersurface $Y_{i,j}$ is periodic and moreover, the length of its period under the coordinatewise action on $(\bP^1)^n$ of the rational functions $f_i$ is bounded solely in terms of the degrees of the rational functions $f_i$ and also in terms of $p$, $e$ and $[\kappa:\Fp]$. Since $Y$ is an irreducible component of $\bigcap_{i,j}Y_{i,j}$, this  concludes the proof of Theorem~\ref{main result disintegrated}.
\end{proof}

\begin{theorem}
\label{main result trivial}
Theorem~\ref{main result} holds under the assumption that each rational function $f_i$ has degree equal to $1$.
\end{theorem}

\begin{proof}
  We use Poonen's \cite{Poonen} result on $p$-adic interpolation of
  iterates.  Since each $f_i$ is a degree one map with good reduction
  at $\pi$, there is an $M$ depending only on $p$, $e$, and
  $[\kappa: \F_p]$ such that $f_i^M(x) \equiv x \pmod{p^c}$, where $c$ is
  the smallest integer greater than $1/(p-1)$.  Now, let
  $\gamma = (\gamma_1 \dots, \gamma_n)$ be any point in $Y(\Kbar)$,
  and let $R$ be the ring of $p$-adic integers in $K(\gamma)$.  Then
  taking the local power series expansion for $f_i^M$ at $\gamma_i$
  gives a power series $g_i \in R [ x ]$, convergent on $R$, such that
  $g_i(z) = f_i^M(z)$ for all $z  \in R$.  Since $f_i^M(x) \equiv x
  \pmod{p^c}$, we have $g_i(x) \equiv x \pmod {p^c}$.  

  By \cite{Poonen}, for each $i$, there is therefore a power series
  $\theta_i \in R [ x ]$, convergent on $R$, such that
  $\theta_i(k) = (f_i^M)^k(\gamma_i)$ for all positive integers $k$.
  Now, let $H$ be any polynomial in the vanishing ideal of $Y$,
  dehomogenized at $\gamma$.  Since there are infinitely many $k$ such
  that $\Phi^{kM}(\gamma) \in Y$ (because $Y$ is periodic under
  $\Phi$), there are infinitely many $k$ such that
  $H(\theta_1(k), \dots, \theta_n(k)) = 0.$ Since a convergent
  $p$-adic power series has finitely many zeros unless the series is
  identically zero, it follows that
  $H(\theta_1(k), \dots, \theta_n(k)) = 0$ for all $k$, so
  $\Phi^{Mk}(\gamma) \in Y$ for all $k$.  Since this is true for all
  $\gamma \in Y(\Kbar)$, we must have that $Y$ has period dividing $M$.

As an aside, we could argue also directly, using the fact that each $f_i$ is a linear map to find explicit formulas for $f_i^n(\alpha_i)$ and thus construct explicit $p$-adic analytic functions as above, thus proving the desired conclusion in Theorem~\ref{main result trivial}.
\end{proof}

\begin{theorem}
\label{main result groups}
Theorem~\ref{main result} holds under the assumption that each rational function $f_i$ has degree larger than $1$ and is conjugate either to $x^{\pm \deg(f_i)}$, or to $\pm C_{\deg(f_i)}$, or to a Latt\'es map.
\end{theorem}

\begin{proof}
Arguing as in the proof of Theorem~\ref{main result trivial}, at the expense of replacing each $f_i$ by a conjugate (which does not change the conclusion of our result), we may assume each $f_i(x)$ is either a Latt\'es map, or equal to $\pm C_{\deg(f_i)}$, or equal to $x^{\pm \deg(f_i)}$.

In this case there exist elliptic curves $E_1,\dots, E_{k}$ (for some $0\le k\le n$) and there exists an endomorphism $\Psi$ of $S:=\bG_m^{n-k}\times \prod_{i=1}^{k}E_i$ along with a finite morphism $\eta:S\lra (\bP^1)^n$ such that 
\begin{equation}
\label{conjugation upstairs}
\eta\circ \Psi = \Phi\circ \eta.
\end{equation} 
Indeed, for each Latt\'es map $f_i$ (without loss of generality, we assume $f_{n-k+1},\dots, f_n$ are all the Latt\'es maps), we know there exists an elliptic curve $E_i$ along with some endomorphism $g_i$ and also there exists some morphism $\eta_i:E_i\lra \bP^1$ of degree $2$ (identifying each point $P$ of $E_i$ with $-P$) such that $\eta_i\circ g_i = f_i\circ \eta_i$ for each $i=n-k+1,\cdots, n$. Furthermore, assuming $f_1,\dots, f_\ell$ (for some $0\le\ell\le n-k$) are all the rational functions conjugated to $\pm C_{\deg(f_i)}$, then for each $i=1,\dots, \ell$, we let $\eta_i:\bG_m\lra \bP^1$ be defined by $\eta_i(x)=x+\frac{1}{x}$; also, we let $g_i(x)=\pm x^{\deg(f_i)}$ for each such $i$. Finally, for each $i=\ell+1,\dots, n-k$, we let $g_i=f_i$ (which is thus equal to $ x^{\pm \deg(f_i)}$) and also let $\eta_i(x)=x$ for each such $i=\ell+1,\dots, n-k$. Then  \eqref{conjugation upstairs} holds with $\eta:=(\eta_1,\dots, \eta_n)$ and also with $\Psi:=(g_1,\dots, g_n):S\lra S$ (acting coordinatewise). 

Now, $Y\subset (\bP^1)^n$ is periodic under the action of $\Phi$ if and only if an irreducible component of $Z:=\eta^{-1}(Y)$ is periodic under the action of $\Psi$; moreover, the two periodic varieties would then have the same length for their corresponding orbits. Hence, it suffices to bound the length of the period for any periodic subvariety of $S$ defined over some given local field. Note that if $Y$ is defined over $K$, then $Z$ is defined over another local field $L$ whose degree over $K$ is bounded by $2^n$ since for each $i=1,\dots, \ell$ and also for each $i=n-k+1,\dots, n$, the degree of $\eta_i$ is $2$ (while $\deg(\eta_i)=1$ if $i=\ell+1,\dots, n-k$).  Therefore, replacing $K$ by $L$ simply increases the size of the residue field of the new local field and also increases its ramification index (over $\Q_p$), but both the size of the residue field of $K$ and the ramification index of $K$ increase by a scalar factor which depends only on $n$.

Furthermore, at the expense of conjugating each $g_i(x)$ (for $1\le i\le \ell$) by a suitable linear transformation $x\mapsto \zeta_i x$, we may actually assume each $g_i(x)=x^{\deg(f_i)}$ (i.e., modulo a conjugation by a linear scaling map  given by a root of unity, we may assume the sign infront of each monomial from the definition of $g_i(x)$ is positive). In order to replace the original maps by the new maps, we might need to replace $K$ by a finite extension, which depends only on the degrees of the original rational functions $f_i$.

We observe that $S$ is a split semiabelian variety (whose abelian part is a product of elliptic curves). Thus the periodic subvarieties of $S$ (under the coordinatewise action of group homomorphisms on each of its $1$-dimensional factors) are torsion translates of algebraic subgroups (a more general result holds inside any semiabelian variety, as proven in \cite[Lemme~10]{Hindry}). Moreover, noting that there are no nontrivial morphisms between a torus and an elliptic curve, then each periodic subvariety $Z$ of $S$ is the zero locus of finitely many equations of the form 
\begin{equation}
\label{form equation torus}
\prod_{i=1}^{n-k} x_i^{c_i} = \zeta
\end{equation} 
for some root of unity $\zeta\in K$ and some integers $c_i$, or 
\begin{equation}
\label{form equation elliptic}
\sum_{i=1}^k \psi_i(y_i)=Q,
\end{equation}
where $E$ is one of the elliptic curves $E_j$ (for some $j=1,\dots, k$) and $\psi_i:E_i\lra E$ are group endomorphisms, while $Q$ is a torsion point of $E(K)$. Clearly, if $E_i$ and $E$ are not isogenuous, then $\psi_i$ is the trivial map. 

Furthermore, the algebraic subgroup of $\bG_m^{n-k}$ defined by the equations 
\begin{equation}
\label{form equation torus 2}
\prod_{i=1}^{n-k} x_i^{c_i} = 1
\end{equation} 
(obtained by replacing each $\zeta$ by $1$ in \eqref{form equation torus}) must be periodic under the coordinatewise action of the $g_i$'s. Then we note that for each equation \eqref{form equation torus 2}, letting $I$ be the set consisting of all indices $i$ such that $c_i\ne 0$, we have that $\deg(g_i)$ is the same for each $i\in I$, or otherwise the hypersurface given by the equation \eqref{form equation torus 2} is not periodic under the coordinatewise action of the $g_i$'s for $i\in I$. So, each hypersurface of the form \eqref{form equation torus 2} is invariant under the coordinatewise action of $g_1^2,\dots, g_{n-k}^2$ on $\bG_m^{n-k}$ since each $g_i$ is a monomial $x^{\pm \deg(g_i)}$ and therefore $g_i^2(x)=x^{\deg(f_i)^2}$; in particular, the algebraic subgroup which is the zero locus of all equations \eqref{form equation torus 2} is fixed  by the coordinatewise action of the $g_i^2$'s.

Similarly, the algebraic subgroup of $E_1\times \cdots \times E_k$ given by the equations
\begin{equation}
\label{form equation elliptic 2}
\sum_{i=1}^k \psi_i(y_i)=0,
\end{equation}
(obtained by replacing each $Q$ by $0$ in \eqref{form equation elliptic}) must be  periodic under the coordinatewise action of $g_{n-k+1},\dots, g_n$. We claim that  the length of its period is bounded by the number of roots of unity contained in the endomorphism ring of $E$. Indeed, whenever $\psi_i$ is nontrivial, then $E_i$ and $E$ have isomorphic endomorphism rings and moroever, the endomorphism $g_i$ of $E_i$ descends to an endomorphism of $E$. Note that if $\psi_i$ and $\psi_j$ are nontrivial and the endomorphisms $g_{n-k+i}$ and $g_{n-k+j}$ of $E_i$, respectively of $E_j$,  correspond to elements $\omega_i$ and $\omega_j$ in the endomorphism ring $R$ of $E$ whose quotient $\omega_i/\omega_j$ is not a root of unity, then the algebraic subgroup given by equation \eqref{form equation elliptic 2} is not periodic under the action of $(g_{n-k+1},\cdots, g_{n})$. Therefore, the length of the period of the algebraic group given by equation \eqref{form equation elliptic 2} is absolutely bounded because there are at most $6$ roots of unity in an order of an imaginary quadratic number ring (such as the ring $R$).

Using the fact that $Z$ is the zero locus of finitely many equations of the form \eqref{form equation torus} and \eqref{form equation elliptic}, while the algebraic groups given by equations \eqref{form equation torus 2} and \eqref{form equation elliptic 2} have the size of their orbit under $\Psi$ bounded by $6$, we obtain that the length of the orbit of $Z$ under $\Psi$ is bounded in terms of the orders of the roots of unity $\zeta$ appearing in equations of the form \eqref{form equation torus} and also in terms of the orders of the torsion points $Q$ appearing in equations of the form \eqref{form equation elliptic}. 
So, we conclude that the length of the orbit under the action of $\Psi$ of any periodic subvariety of $S$ (defined over the local field $K$) is bounded solely in terms of the size of the group of roots of unity contained in $K$ and also in terms of the size of the torsion subgroups of $E_i(K)$, for $i=1,\dots, k$.  
Since any local field contains finitely many roots of unity (the bound depending solely on the size of its residue field and also depending on its ramification index over $\Q_p$), and also, any elliptic curve has finitely many torsion points over the local field $K$, with an upper bound for the size of its torsion depending solely on the size of the residue field of $K$ and on the ramification index for $K/\Q_p$ (see \cite[Chapter~VII]{Silverman}), then we obtain the desired conclusion in Theorem~\ref{main result groups}.
\end{proof}

Finally, we can prove Theorem~\ref{main result}
\begin{proof}[Proof of Theorem~\ref{main result}.] 
As proven in \cite{Alice} (see also \cite[Theorem~2.30]{MS-Annals} for the case when each $f_i$ is a polynomial), each periodic subvariety $Y$ of the dynamical system  $((\bP^1)^n, \Phi)$ may be written (after a suitable re-ordering of the coordinate axes, and thus of the rational functions acting on them) as $Y_1\times Y_2\times Y_3$, where $Y_i\subset (\bP^1)^{n_i}$ for some integers $n_i$ satisfying $n_1+n_2+n_3=n$ and moreover, the following holds:
\begin{itemize}
\item the rational functions $f_1,\cdots, f_{n_1}$ are all disintegrated;
\item each rational function $f_i(x)$ from the list $f_{n_1+1},\cdots, f_{n_1+n_2}$ is conjugate either to $x^{\pm \deg(f_i)}$, or to $\pm C_{\deg(f_i)}$, or to a Latt\'es function (and its degree is greater than $1$); 
\item the rational functions $f_{n_1+n_2+1},\cdots, f_{n_1+n_2+n_3}$ have degree equal to $1$;
\item $Y_1$ is periodic under the action of $\Phi_1:=(f_1,\cdots, f_{n_1})$ on $(\bP^1)^{n_1}$, while $Y_2$ is periodic under the action of $\Phi_2:=(f_{n_1+1},\cdots, f_{n_1+n_2})$ and $Y_3$ is periodic under the action of $\Phi_3:=(f_{n_1+n_2+1},\cdots, f_{n_1+n_2+n_3})$. 
\end{itemize}
Therefore, the uniform bound on the length of the orbit of $Y$ under the the action of $\Phi$ follows using Theorems~\ref{main result disintegrated},~\ref{main result trivial}~and~\ref{main result groups}, which provide uniform bounds for the length of the orbits of each of the periodic subvarieties $Y_j$ (for $1\le j\le 3$). 
\end{proof} 

We conclude this section by proving Theorem~\ref{main result smooth point}.
\begin{proof}[Proof of Theorem~\ref{main result smooth point}.]
  We argue by induction on $n$; the case $n=1$ follows from
  \cite{mike-thesis}. So, we assume now that $n\ge 2$ and that
  Theorem~\ref{main result smooth point} holds for endomorphisms of
  $(\bP^1)^{n-1}$. In particular, we may assume that $Y$ projects
  dominantly onto each coordinate axes, since otherwise
  $Y= Y_0\times \{\gamma\}$, where $\gamma\in\bP^1_K$ is a periodic
    point (of bounded period, according to \cite{mike-thesis}) and
    then the inductive hypothesis yields the desired conclusion.

Let $\cY$ be a model for $Y$ over $\fo_K$ such that $\cY$ is closed in
$(\bP^1_{\fo_K})^n$.  Then $\cY$ is a projective $\fo_K$-scheme, so
every point in $Y(K)$ extends to a unique point in $\cY(\fo_K)$.  We let
$\alpha$ denote the point in $\cY(\fo_K)$ to which $(\alpha_1, \dots,
\alpha_n)$ extends.

Let $\bar{\alpha}$ be the reduction of $\alpha$ modulo $p$, i.e., the
intersection of $\alpha$ with the special fiber of $\cY$. Since the
intersection of $\alpha$ with the generic fiber $Y$ of $\cY$ is a
smooth point $x$, then \cite[Proposition~2.2]{BGT-IMRN} yields the
existence of a Zariski dense,  uncountable set $U\subset \cY(\fo_v)$
consisting of sections $\beta$ whose intersection with the special
fiber of $\cY$ equals $\bar{\alpha}$. The fact that $U$ is uncountable
follows from the $p$-adic implicit function theorem used in the proof
of \cite[Proposition~2.2]{BGT-IMRN}, which yields the existence of a
$p$-adic submanifold of $\cY$ whose points all have the same reduction modulo $p$. More precisely, as shown in \cite[Proposition~2.2]{BGT-IMRN}, letting $d:=\dim(Y)$, at the expense of relabelling the coordinate axes, we know that for a sufficiently small $p$-adic neighborhood $\cU$ of $(\alpha_{n-d+1},\dots, \alpha_n)$ (where $\alpha:=(\alpha_1,\dots, \alpha_n)$), we have $p$-adic analytic functions $\cF_1,\dots, \cF_n$ on $\cU$ such that for each $\gamma\in\cU$, the point $(\cF_1(\gamma),\dots, \cF_n(\gamma))$ is on $\cY$. 

Each one of the points $\beta\in U$ is of the form $(\beta_1,\dots, \beta_n)$ and if for some point $\beta$, we have that each $\beta_i$ is not preperiodic for the action of $f_i$, then the result follows from Theorem~\ref{main result}. So, assume for each $\beta\in U$, there exists some $i\in\{1,\dots, n\}$ such that $\beta_i$ is preperiodic for $f_i$. Therefore there must exists some $i_0\in\{1,\dots, n\}$ such that there exists an uncountable set $U_0\subset \cY(\fo_v)$ consisting of points $\beta$ with the property that the corresponding $\beta_{i_0}$ is preperiodic for the action of $f_{i_0}$. Furthermore, we may assume that there exists no submanifold $\cU_0$ of $\cU$ of dimension less than $d$ such that for each $\beta\in U_0$, there exists some $\gamma\in \cU_0$ such that $\beta=(\cF_1(\gamma), \cdots, \cF_n(\gamma))$.

For the sake of simplifying the notation, we identify each $\beta_{i_0}$ with its intersection with the generic fiber of $\bP^1_{\fo_v}$ and thefore, view each $\beta_{i_0}$ as a point in $\bP^1(K)$ whose reduction modulo $p$ is $\overline{\alpha_{i_0}}$. There are two cases.

{\bf Case 1.} The set $\{\beta_{i_0}\colon \beta\in U_0\}$ is uncountable.

Then, according to our assumption, we obtain that $f_{i_0}$ has an uncountable set of preperiodic points, which automatically yields that $f_{i_0}$ itself must be preperiodic; more precisely, $\deg f_{i_0}=1$ (i.e, $f_{i_0}$ is an automorphism) and $f_{i_0}^{\ell}$ is the identity, for some positive integer $\ell$. So, at the expense of replacing $\cY$ by $\cY^\ell$, we may assume $\cY$ is fixed under the action of $f_{i_0}$ on the $i_0$-th coordinate axis. Therefore, the length of the period of $\cY$ equals the length of the period of $\cY_0$ under the coordinatewise action of the rational functions $f_i$ for $i\ne i_0$, where $\cY_0$ is the projection of $\cY$ on the remaining $(n-1)$ coordinate axes, other than the $i_0$-th coordinate axis. In this case, we are done by the inductive hypothesis; also, note that $\ell$ depends solely on $p$, $e$ and $[\kappa:\Fp]$. Indeed, the linear map $f_{i_0}(x)$ is conjugated over a quadratic extension $L$ of $K$ to one of these two maps $x\mapsto x+1$ or $x\mapsto ax$; furthermore, since $f_{i_0}$ has finite order, then $f_{i_0}(x)$ must be conjugated to a map $x\mapsto ax$, where $a$ is a root of unity contained in $L$. Because the number of roots of unity in a local field $L$ is bounded solely in terms of $p$ and of the degree $[L:\Q_p]$, our claim follows. 

{\bf Case 2.} The set $\{\beta_{i_0}\colon \beta\in U_0\}$ is countable.

Let $\cU_0\subset \cU$ be the set consisting of all points $\gamma\in \cU$ such that $(\cF_1(\gamma),\dots, \cF_n(\gamma))\in U_0$. Then we know that the set 
$\left\{\cF_{i_0}(\gamma)\colon \gamma\in \cU_0\right\}$ is countable; furthermore, by our assumption, we know that $\cU_0$ is an uncountable set, which is not contained in a submanifold of $\cU$ of dimension less than $d$. Therefore, $\cF_{i_0}$ must be constant, and thus the projection of $\cY$ on the $i_0$-th coordinate axis must be constant, contrary to our hypothesis from the beginning of our proof. This concludes the proof of Theorem~\ref{main result smooth point}.     
\end{proof}


\end{document}